\documentclass{amsart}
\usepackage{amsmath,amsthm,amssymb,amsfonts}
\usepackage{theoremref}
\ifx\pdfoutput\undefined \usepackage[hypertex]{hyperref} \else \usepackage[pdftex,pdfstartview=FitH,plainpages=false,pdfpagelabels]{hyperref} \fi
\numberwithin{equation}{section}
\makeindex

\newcommand\xappa\kappa
\newcommand\yota\iota
\newcounter{consta}

\newcounter{constb}

\newcounter{constc}[section]

\DeclareFontFamily{OT1}{rsfs}{}
\DeclareFontShape{OT1}{rsfs}{n}{it}{<-> rsfs10}{}
\DeclareMathAlphabet{\mathscr}{OT1}{rsfs}{n}{it}
\swapnumbers
\newtheorem{thm}{Theorem}[section]
\newtheorem{lem}[thm]{Lemma}

\newtheorem{prop}[thm]{Proposition}
\newtheorem*{lem*}{Lemma}
\newtheorem*{thm*}{Theorem}
\newtheorem*{conj*}{Conjecture}
\newtheorem*{prop*}{Proposition}
\newtheorem{defn*}{Definition}

\newtheorem{ex}[thm]{Example}
\newtheorem{cor}[thm]{Corollary}

\theoremstyle{definition}
\newtheorem{defn}[thm]{Definition}
\theoremstyle{remark}
\newtheorem{rem}[thm]{Remark}
\newtheorem{obs}[thm]{Observation}
\newtheorem*{obs*}{Observation}
\newtheorem*{rem*}{Remark}
\theoremstyle{definition}\newtheorem*{acknowledgments}{Acknowledgments}

\begin{document}
\title[Generalizations of Furstenberg's Diophantine result]{Generalizations of Furstenberg's Diophantine result}
\author{Asaf Katz}
\address{Einstein Institute of Mathematics, The Hebrew University of Jerusalem, Jerusalem,
91904, Israel.}
\email{asaf.katz@mail.huji.ac.il}

\thanks{The research was supported by ERC grant (AdG Grant 267259) and ISF grant (983/09).}
\date{}

\begin{abstract} We prove two generalizations of Furstenberg's Diophantine result regarding density of an orbit of an irrational point in the one-torus under the action of multiplication by a non-lacunary multiplicative semi-group of $\mathbb{N}$.
We show that for any sequences $\{a_{n} \},\{b_{n} \}\subset\mathbb{Z}$ for which the quotients of successive elements tend to $1$ as $n$ goes to infinity, and any  infinite sequence $\{c_{n} \}$, the set $\{a_{n}b_{m}c_{k}x : n,m,k\in\mathbb{N} \}$ is dense modulo $1$ for every irrational $x$.
Moreover, by ergodic-theoretical methods, we prove that if $\{a_{n} \},\{b_{n} \}$ are sequence having smooth $p$-adic interpolation for some prime number $p$, then for every irrational $x$, the sequence $\{p^{n}a_{m}b_{k}x : n,m,k\in\mathbb{N} \}$ is dense modulo~1.

\end{abstract}
\maketitle

\section{Introduction}
In his seminal paper  \cite{furstenberg-disjointness}, H. Furstenberg proved the following result -
\begin{thm*}[Furstenberg's Diophantine result] Assume that $a,b$ are two multiplicatively independent integers, then for any irrational $x$, the set $\{a^{n}b^{m}x: n,m\in\mathbb{N}\}$ is dense in $\mathbb{R}/\mathbb{Z}$.
\end{thm*}
In \cite{boshernitzan-furstenberg-result}, M. Boshernitzan gave an elementary proof of Furstenberg's result.
Moreover, Furstenberg conjectured the following conjecture, which is still open -
\begin{conj*}[Furstenberg's $\times a,\times b$ conjecture] Assume that $a,b$ are two multiplicatively independent integers, and assume that $\mu$ is a Borel probability measure on $\mathbb{R}/\mathbb{Z}$, which is invariant under the maps $T_{a}(x)=ax \bmod 1,T_{b}(x)=bx \bmod 1$, and ergodic, then $\mu$ must be the Haar measure or atomic measure.
\end{conj*}

The strongest result up to date towards Furstenberg's conjecture is the following theorem by D. Rudolph and A. Johnson -
\begin{thm*}[\cite{johnson1992measures}] In the settings of Furstenberg's conjecture, if the measure $\mu$ has \emph{positive entropy} with respect to $T_{a}$, then $\mu$ is the Haar measure.
\end{thm*}
Rudolph and Johnson proved their measure classification result directly by using symbolic dynamics.
By different methods, B. Host established a related pointwise result in \cite{host1995nombres}, which implies the measure classification result.

Following Host's ideas, D. Meiri defined the following class of sequences -
\begin{defn*}
A sequence of integers $\{a_{n} : n\in\mathbb{N}\}$ is a \emph{$p$-Host} sequence, if for every Borel probability measure $\mu$ on $\mathbb{R}/\mathbb{Z}$ which is invariant under the map $T_{p}(x)=p\cdot x\bmod(1)$,ergodic, and has positive entropy with respect to $T_{p}$, then for $\mu$-a.e. $x$, the set $\{a_{n}x : n\in \mathbb{N}\}$ is equidistributed in $\mathbb{R}/\mathbb{Z}$.
\end{defn*}
By this definition, Host's theorem is essentially the following statement -
\begin{thm*}[Host]
If $p$ and $q$ are co-prime integers, then the sequence $\{q^{n}\}$ is a $p$-Host sequence.
\end{thm*}

In \cite{meiri}, D. Meiri  explicitly constructed $p$-Host sequences by studying the $p$-adic distribution properties of $p$-adic analytic functions. As a consequence, he derived a simple way of producing $p$-Host sequences via $p$-adic interpolation.

\begin{defn*}
We say that a sequence of integers $\{a_{n}\}$ has \emph{$p$-adic interpolation} if there exists a continuous function $f(x)$ defined in the unit disk of some finite extension of $\mathbb{Q}_{p}$ with $a_{n}=f(n)$.
\end{defn*}

We remark that E. Lindenstrauss expanded Meiri's work in \cite{lindenstrauss2001p}, and B. Host has also used $p$-adic formalism in \cite{host2000some}.

Recently in \cite{BLMV}, Bourgain-Lindenstrauss-Michel-Venkatesh gave an \emph{effective} proof of the Rudolph-Johnson theorem and as a result, Furstenberg's Diophantine result, using inequalities regarding linear forms in logarithms to give a quantitative form of the non-lacunarity which occurs in the semigroup $\langle p,q\rangle \subset\mathbb{N}$.

In this paper we focus on generalizing Furstenerg's density result, by considering only 1-parameter-action with some extra sequence which enable us to construct some invariant sets or measures with ''positive dimensionality''. Afterward, building upon the works by Boshernitzan and Meiri, we use the third sequence to conclude equidistribution or density.

Our first theorem is topological in nature and does not rely on measure-classification techniques.
\begin{thm}\label{thm:non-lacunary}\hypertarget{thm:non-lacunary}
Assume that $\{a_{n}\},\{b_{m}\}\subset\mathbb{N}$ are increasing sequences satisfying $\lim_{n\to\infty} a_{n+1}/a_{n}=1$, $\lim_{m\to\infty} b_{m+1}/b_{m}=1$ and $\{c_{k}\}$ is any unbounded sequence of integers, then for every irrational $x$, the set $\{a_{n}b_{m}c_{k}x : n,m,k\in\mathbb{N}\}$ is dense modulo $1$.
\end{thm}
The proof of this theorem follows in the spirit of Furstenberg's own proof, the extra information needed to conclude the theorem relies on a result of Boshernitzan~\cite{boshernitzan-dilations-of-nonlacunary-sequences}.

By combining the approach demonstrated in \cite{BLMV} towards the proof of Furstenberg's result and the results by Meiri and Lindenstrauss regarding $q$-Host sequences, we can deduce the following theorem -
\begin{thm}\label{thm:p-adic} \hypertarget{p-adic}
Fix $q \in \mathbb{N}$ and let $\{a_{m} : m \in \mathbb{N}\}, \{b_{k} : k\in\mathbb{N}\}$ be two sequences of integers satisfying the following conditions:
\begin{itemize}
	\item [Positive Entropy] There exists some prime number $p$ which divides $q$ such that $\{a_{m}\}$ admits a smooth $p$-adic interpolation with only finitely many critical points inside the unit disc.
	\item [$q$-Host] There some prime number $p$ which divides $q$ for which $\{b_{k}\}$ admits a smooth $p$-adic interpolation with only finitely many critical points inside the unit disc, and for every other prime $p'$ which divides $q$ for which $\{b_{k}\}$ does not admit a smooth $p'$-adic interpolation with finitely many critical points inside the unit disc, $\|b_{k}\|_{p'} \to 0$, where $\| \cdot \|_{p'}$ stands for the $p'$-adic norm.
\end{itemize}
Then for any $x\in\mathbb{R}\backslash\mathbb{Q}$ the set $\{q^{n}a_{m}b_{k}x : n,m,k\in\mathbb{N}\}$ is dense in $\mathbb{R}/\mathbb{Z}$.
\end{thm}

Various other generalizations of Furstenberg's result have been considered previously. \
Call a sequence of integers $\{a_{n}\}$ \emph{universally densifying sequence} if for every $\alpha \notin \mathbb{Q}$ the set $\{a_{n}\cdot\alpha\}$ is dense modulo $1$.
A famous theorem of Hardy and Littlewood (generalizing a lemma due to Kronecker) shows that $\{p(n)\}$ is a universally densifying sequence, where $p(n)$ is any non-constant polynomial with integer coefficients.
Furstenberg's result is equivalent to saying that a sequence of integers forming a non-lacunary multiplicative semigroup is a universally densifying sequence.
In \cite[Theorem~$1.2$]{kra1999generalization}, B. Kra proved that coupled sums of products formed from powers of multiplicatively independent integers form a universally densifying sequence, namely sets of the form $\left\{\left(\sum_{i=1}^{k}p^{n}_{i} q^{m}_{i}\right) \cdot \alpha\right\}$ are dense modulo $1$. Moreover, in \cite[Corollary~$2.2$]{kra1999generalization} Kra proves finer density properties of the sequence of $\{2^{n}3^{m}\alpha\}$ for irrational number $\alpha$, showing that adding arbitrary shifts depending on $n$ to this sequence results in a dense sequence as well.

In \cite{Gorodnik20122499}, A. Gorodnik and S. Kadyrov generalized Kra's result to higher-dimensional settings.
In \cite{urban2008sequences,urban2008algebraic} R. Urban generalized Kra's result by considering powers of multiplicatively-independent algebraic integers.

Our Theorem~\ref{thm:p-adic} is used in corollary~\ref{cor:universally-densifying} to provide new classes of universally densifying sequences, e.g. the sequence $\left\{2^{n}3^{3^{\cdot^{\cdot^{\cdot^{3^{p_{1}(m)}}}}}}3^{3^{\cdot^{\cdot^{\cdot^{3^{p_{2}(k)}}}}}} : n,m,k\in\mathbb{N}\right\}$, where $p_{1},p_{2}$ are polynomials with integer coefficients is universally densifying sequence, providing examples of very sparse universally densifying sequences. This answers a question by Y. Bugeaud.

\begin{acknowledgments}
The results of this paper were obtained as part of the author's M.Sc thesis at the Hebrew University
of Jerusalem under the guidance of Prof. Elon Lindenstrauss, to whom I am grateful for introducing me to the problem and many helpful discussions and insights. The author also wishes to thank Uri Shapira and Peter Varju for useful comments about a preliminary version of this paper.
\end{acknowledgments}

\section{Proof of Theorem~\ref{thm:non-lacunary}.} \label{sec:non-lacunary}
We begin the proof with the following easy lemma, essentially due to Furstenberg, which we include for the sake of completeness -
\begin{lem}\label{lem:density-of-non-lacunary} Let $\{a_{n} : n\in \mathbb{N}\}\subset\mathbb{N}$ be a sequence of increasing integers satisfying $\lim_{n\to \infty}a_{n+1}/a_{n}=1$, and $\{d_{m} : m\in\mathbb{N}\}\subset\mathbb{R}/\mathbb{Z}$ be a sequence such that $0$ is a non-isolated point in the closure of $\{d_{m}\}$, then the set $\{a_{n}d_{m} : n,m\in\mathbb{N} \}$ is dense modulo $1$.
\end{lem}
\begin{proof}
We show that the sequence is $\varepsilon$-dense, for any $\varepsilon>0$.
By the non-lacunarity of $\{a_{n}\}$ there exists some $n_{0}$ such that for all $n>n_{0}$, $a_{n+1}-a_{n}<\varepsilon a_{n}$.
Define $x_{0}$ to be a point in $\{d_{n}\}\cap(0,\varepsilon/a_{n_{0}+1})$.
Define $S=\{a_{n_{0}+1},\ldots,a_{m}\}$ to be all the elements which are larger than $a_{n_{0}}$ but smaller than $1/x_{0}$.
For all $n_{0}\leq i \leq m$ we have - $a_{i+1}x_{0}-a_{i}x_{0}\leq (a_{i+1}-a_{i})x_{0}\leq \varepsilon a_{i}x_{0}$.
Hence we deduce that the set $\{Sx_{0}\}$ is $\varepsilon$-dense in $\mathbb{R}/\mathbb{Z}$.
By the choice of $x_{0}$, we have that $a_{i}x_{0}\leq 1$ for any $a_{i}\in S$.
\end{proof}

\begin{defn}
For any sequence $\{a_{n} : n \in\mathbb{N}\}\subset\mathbb{N}$, define the \emph{exceptional set} $E_{\text{density}}(\{a_{n}\})$ to be all the numbers $x\in\mathbb{R}/\mathbb{Z}$ so that $\{a_{n}x : n \in \mathbb{N}\}$ is not dense modulo 1.
For an interval $I \subset\mathbb{R}/\mathbb{Z}$ define the exceptional set $E^{I}_{\text{density}}(\{a_{n}\})$ to be all the numbers $x\in\mathbb{R}/\mathbb{Z}$ so that $\{a_{n}x : n \in \mathbb{N} \} \cap I =\emptyset$.
\end{defn}

We recall the definitions of \emph{upper box dimension} and \emph{upper packing dimension} which will be used in the following proof.
\begin{defn}[Upper box dimension]
Let $A$ be a non-empty Borel bounded subset of an Euclidean space.
Denote by $N_{\delta}(A)$ the smallest number of sets of diameter at-most $\delta$ which cover $E$.
We define the \emph{upper box dimension} as
\begin{equation*}
\overline{dim}_{box}(E)=\underset{\delta \rightarrow 0}{\overline{lim}} -\frac{log(N_{\delta}(E))}{log(\delta)}.
\end{equation*}
This dimension is also known as the \emph{upper Minkowski dimension}.
\end{defn}

\begin{defn}[Upper packing dimension]
Let $A$ be a Borel subset of an Euclidean space, we define the \emph{upper packing dimension} of $A$ to be
\begin{equation*}
\overline{dim}_{p}(A)=\inf\left\{\sup_{i} \overline{dim}_{box}(A_{i}) \middle\vert A=\bigcup_{i=1}^{\infty}A_{i}, A_{i} \text{ is bounded} \right\}.
\end{equation*}
\end{defn}
\begin{rem}
Obviously, $\overline{dim}_{p}(A)\leq \overline{dim}_{box}(A)$.
\end{rem}

The following result is due to Boshernitzan.
\begin{thm}[\cite{boshernitzan-dilations-of-nonlacunary-sequences}, Theorem~$1.3$]\label{boshernitzan-exceptional} If $\{b_{m} : m\in \mathbb{N}\}$ is a sequence satisfying $\lim_{m\to\infty} b_{m+1}/b_{m}=1$ then for any open interval $I\subset\mathbb{R}/\mathbb{Z}$
$$\overline{\dim}_{box}E^{I}_{\text{density}}(\{b_{m}\})=0.$$
\end{thm}

As a result one deduces that $\overline{\dim}_{p}E_{\text{density}}(\{b_{m}\})=0$.
We also need the following observation -
\begin{obs}\label{obs:product-hausdorff}
Let $A$ be a Borel set, and assume that $\dim_{H}(A\times A)\geq c$ for some constant $c$. Then $\overline{\dim}_{p}(A)\geq \frac{c}{2}$, where $\dim_{H}$ denotes the Hausdorff dimension.
\end{obs}
This observation follows from the product properties of the Hausdorff dimension for Borel sets (c.f. \cite{mattila}, Theorem $8.10$), namely
\begin{equation*}
\dim_{H}(A\times B)\leq \dim_{H}(A)+\overline{\dim}_{p}(B).
\end{equation*}

Now we prove Theorem~\ref{thm:non-lacunary}.
\begin{proof}[Proof of Theorem~\ref{thm:non-lacunary}]
Let $x$ be an irrational number. Define the difference set $D=\overline{\{c_{k}x\}}-\overline{\{c_{k}x\}}$.
Because $\{c_{k} : k \in \mathbb{N}\}$ is an unbounded sequence of integers, the difference set $D$ contains zero as a non-isolated point.
By lemma~\ref{lem:density-of-non-lacunary}, the set $\{b_{m}D : m \in \mathbb{N}\}$ is dense modulo 1.
By observation~\ref{obs:product-hausdorff}, the set $\overline{\{b_{m}c_{k}x : m,k\in \mathbb{N}\}}$ has upper packing dimension larger than $1/2$, and in particular $\overline{dim}_{p}(\overline{\{b_{m}c_{k}x : m,k\in\mathbb{N}\}})>0$.
By \cite[Theorem~1.3]{boshernitzan-dilations-of-nonlacunary-sequences}, the closure of $\{b_{m}c_{k}x\}$ contains a point outside the exceptional set $E_{\text{density}}(\{a_{n}\})$, hence the result follows.
\end{proof}

\section{Construction of a $T_{q}$ invariant measure with positive entropy.}\label{measure}

We start with some definitions.
\begin{defn}[$T_{q}$ map]
Given an integer $q$, the $\times q$ map on the torus is defined as
\begin{equation*}
T_{q}(x)=q\cdot x \text{ mod }1.
\end{equation*}
For a real number $x$ we denote the distance from $x$ to the nearest integer by $\|x\|$.
\end{defn}
\begin{defn}[Combinatorial Entropy]
Let $\{a_{k} : k \in \mathbb{N} \}$ be a sequence of integers. For an integer $\ell \in \mathbb{N}$, define $a_{k}^{(\ell)}$ to be $a_{k}^{(\ell)}=a_{k}\bmod(\ell)$.
Define the \emph{$\ell^{n}$ combinatorial entropy} of the sequence $\{a_{k} : k\in\mathbb{N}\}$, $h_{\text{comb }\ell^{n}}(\{a_{k}\})$ by 
$$ h_{\text{comb }\ell^{n}}(\{a_{k}\})=\frac{\log\left\lvert \left\{ a_{k}^{(\ell^{n})} : k \in \mathbb{N} \right\} \right\rvert}{n}. $$
We define the \emph{$\ell$-adic combinatorial entropy} to be
$$ h_{\ell-\text{adic}}(\{a_{k}\})=\lim_{n\to\infty} h_{\text{comb }\ell^{n}}(\{a_{k}\}), $$
provided this limit exists.
Similarly, we define the \emph{upper $\ell$-adic combinatorial entropy} to be
$$ \overline{h_{\ell-\text{adic}}}(\{a_{k}\})=\overline{\lim_{n\to\infty}} h_{\text{comb }\ell^{n}}(\{a_{k}\}). $$
\end{defn}

\begin{defn}
Let $q$ be an integer. We say that a sequence $\{a_{n} : n \in \mathbb{N}\}$ has \emph{positive local upper $q$-adic combinatorial entropy} if for some prime numbers $p$ dividing $q$, we have that
$$ \overline{h_{p-\text{adic}}}(\{a_{k}\})>0. $$
\end{defn}

The main construction of this paper appears in the following theorem:
\begin{thm}[Existence of invariant measure]\label{thm:inv-measure-positive-entropy}
Let $q$ be an integer, $x$ an irrational number, and assume that $\{a_{n} : n\in\mathbb{N}\}\subset \mathbb{N}$ has \emph{positive local upper $q$-adic combinatorial entropy}. Then there exists a Borel probability measure $\mu$ on $\mathbb{R}/\mathbb{Z}$ whose support contained in $\overline{\{q^{m}a_{n}x: m,n\in\mathbb{N}\}}$ such that $\mu$ is $T_{q}$-invariant, $T_{q}$-ergodic, and $h_{\mu}(T_{q})>0$.
\end{thm}

\subsection{Difference sets of Cantor-like sets.}\label{cantor-sets}

We begin with a simple observation
\begin{obs}
The dynamical system $(\mathbb{R}/\mathbb{Z},T_{q})$ is \emph{forward-expansive}, meaning that there exists a constant $c>0$ so that for any two distinct points $x,y$, in some iteration in the future, the orbits are at-least $c$ apart, namely there exists $N\geq 0$ so that $d(T_{q}^{N}x,T_{q}^{N}y)>c$.
In particular for the $T_{q}$ map one can choose $c=\frac{1}{q+1}$, as can be seen by representing $x,y\in\mathbb{R}/\mathbb{Z}$ in base-$q$.
\end{obs}

The main result of this section is the following theorem:
\begin{prop}\label{prop:difference-of-rational-points-in-closure}
Let $x$ be irrational number, $q$ a fixed integer, and denote by $\mathcal{O}$ the orbit closure $\overline{\{q^{n}x : n\in\mathbb{N}\}}$. The Minkowski difference $\mathcal{O}-\mathcal{O}$ contains a point of the form $\frac{\ell_{n}}{q^{n}}$ for every $n\in\mathbb{N}$, where $\ell_{n} \in \left\{1,\ldots,q^{n}-1\right\}$. Moreover, we have that $\gcd(\ell_{n},q)$ is the same for all $n\in \mathbb{N}$, and in particular in the case where $q$ is a prime number, we have that $\gcd(\ell_{n},q)=1$.
\end{prop}
Those special rational points will allow us to construct measures with positive entropy.

Now we recall Schwartzman's lemma -
\begin{lem*}[Schwartzman \cite{Schwartzman}]
Let $X$ be a compact metric space. If $T:X\to X$ is a \emph{forward-expansive homeomorphism}, then $X$ is finite.
\end{lem*}

\begin{proof}[Proof of Proposition~\ref{prop:difference-of-rational-points-in-closure}]

Let $\mathcal{O}$ denote the orbit closure of $x$ by the $T_{q}$ map, namely the set $\mathcal{O} =\overline{\{T_{q}^{n}x \}}$ which is infinite by irrationality of $x$, and define the \emph{attractor} of $\mathcal{O}$ to be $\mathcal{O}' =\cap_{i=0}^{\infty} T_{q}^{i}\mathcal{O}$.
By definition $\mathcal{O}'$ is $T_{q}$-invariant set, and by compactness it is non-empty.
Moreover, all the accumulation points of $\mathcal{O}$ are contained inside $\mathcal{O}'$.
We first show that $T_{q}:\mathcal{O}\to\mathcal{O}$ is not injective, we do it by showing that the restriction of $T_{q}$ to $\mathcal{O}'$ itself is not injective.

Assume that $T_{q}:\mathcal{O}' \to \mathcal{O}'$ is injective.
Due to $T_{q}$-invariance, we have that the restriction of the $T_{q}$ map to $\mathcal{O}'$ is surjective and hence a homeomorphism.
By Schwartzman's lemma $\mathcal{O}'$ is a \emph{finite set}, and as a result contained inside a discrete subgroup of $\mathbb{T}$, namely there exists an integer $m$ for which $\mathcal{O}' \subset \{j/m \mid j=0,\ldots,m-1 \}$.

Define the thickening of $\mathcal{O}'$ to be the union of the intervals $I_{j,\varepsilon}$ for $j=0,\ldots,m-1$, where $I_{j/m,\varepsilon}= (j/m-\varepsilon,j/m+\varepsilon)$ for some $\varepsilon < \frac{1}{2m(q+1)}$.
Because $\mathcal{O}'$ is the attractor of the system, there exists some $N_{1} \gg 0$ for which $T_{q}^{N_{1}}\mathcal{O}\subset\cup_{j=0}^{m-1}I_{j/m,\varepsilon}$.
Fix some $x$ in one of the intervals $I_{j,\varepsilon}$. We have that 
\begin{equation*}\text{dist}(T_{q}x,T_{q}(j/m))\leq q\text{dist}(x,j/m)\leq q\varepsilon<1/2m,\end{equation*} where $\text{dist}$ stands for the distance function in the torus, hence we deduce that $T_{q}x\in I_{T_{q}(j/m),\varepsilon}$ as the distance from the interval $T_{q}(j/m)$ to any other interval of the form $I_{j'/m,\varepsilon}$ for any $j'/m \neq T_{q}(j/m)$ is of distance at-least $1/2m$.

Therefore, for any two points $x,y$ which belong to the same interval $I_{j/m,\varepsilon}$ we have that $T_{q}x,T_{q}y$ belong to the interval $I_{T_{q}(j/m),\varepsilon}$, hence their images are less than $2\varepsilon<1/(q+1)$ apart.
By repeating the argument, the whole forward orbits of $x$ and $y$ under the $T_{q}$ map are close, namely $\sup_{n\in\mathbb{N}} \text{dist}(T_{q}^{n}x,T_{q}^{n}y)<\frac{1}{q+1}$, which contradicts forward-expansiveness of the $T_{q}$ map on the whole torus.

Thus there exists two points $y_{1},y_{1}'\in\mathcal{O'}$ with $T_{q}y=T_{q}y'$ or equivalently by the definition of the $T_{q}$ map $y-y'=\ell_{1}/q$, for some integer $\ell_{1}$ with $0<\ell_{1}<q$.
Define the following sequences of points $y_{n}=T_{q}^{-1}y_{n-1},\ y'_{n}=T_{q}^{-1}y'_{n-1}$ so we have $y_{n},y'_{n}\in\mathcal{O}'$ for all $n\in\mathbb{N}$ and moreover $y_{n}-y'_{n}=\ell_{n}/q^{n}$ for some $0<\ell_{n}<q^{n}$ as $T_{q}^{n}y_{n}=T_{q}^{n}y'_{n}$ by construction.

Moreover, $\ell_{n-1}/q^{n-1}=q\cdot\ell_{n}/q^{n} \bmod(1)$ as $T_{q}(y_{n}-y'_{n})=y_{n-1}-y'_{n-1}$ and by induction we get $\ell_{1}/q = q^{n-1}\cdot\ell_{n}/q^{n} = \ell_{n}/q \bmod(1)$, showing that $\gcd(\ell_{n},q)=\gcd(\ell_{1},q)$ for all $n\in \mathbb{N}$.
\end{proof}

The proof of Theorem~\ref{thm:inv-measure-positive-entropy} is divided into two cases - when $q$ is a prime number, and when $q$ is composite.
In $\S$\ref{sub-prime-theorem} we prove the theorem for the prime case, and in $\S$\ref{sub-composite-theorem} we show the modifications needed in order to handle the composite case.

\subsection{Proof of Theorem~\ref{thm:inv-measure-positive-entropy} - prime case.}\label{sub-prime-theorem}
For the rest of this subsection, fix a prime number $p$ and a sequence $\{ a_{n} : n\in \mathbb{N} \}$ which has positive $p$-adic entropy.

We now utilize the strategy proposed in \cite{BLMV} to produce invariant measures with positive entropy.
Define $P_{n}$ to be the uniform partition of the unit interval into $n$ equally sized sub-intervals.
\begin{obs*}
If $p_{1},p_{2}$ are two atoms of $P_{n}$, then $p_{1}-p_{2}$ is covered by at-most $2$ atoms of $P_{n}$.
\end{obs*}
Define $\Delta$ to be the difference set $\mathcal{O}-\mathcal{O}$ where $\mathcal{O}=\overline{\{T_{p}^{n}x: n\in\mathbb{N}\}}$ is the forward-orbit of $x$ under $T_{p}$.
By proposition~\ref{prop:difference-of-rational-points-in-closure} we can assume that there are sequences of points $\{x_{N}\},\{y_{N}\} \subset \Delta$ so that $x_{N}-y_{N}=\frac{\ell_{N}}{p^{N}}$ for all $n\in\mathbb{N}$.
By the construction of those points, $\gcd(\ell_{N},p)=1$, hence $\ell_{N}$ is invertible modulo $p^{N}$.
Assume that $\{a_{k} : k\in \mathbb{N}\}$ has \emph{positive upper $p$-adic combinatorial entropy}, then for large enough $N$, $\{a_{k}^{(p^{N})} : k \in \mathbb{N}\}$ has more than $p^{Nc_{1}}$ distinct elements, for some $c_1>0$. Hence the sequence $\{a_{k}(x_{N}-y_{N}) : k \in \mathbb{N}\}$ has at-least $p^{Nc_{1}}$ distinct values modulo $1$.

Denote by $\Delta_{N}$ to be the set of atoms of $P_{p^{N}}$ which $\Delta$ intersects. By the above calculations, $\lvert \Delta_{N} \rvert \geq p^{Nc_{1}}$, because the set $\{a_{k}(x_{N}-y_{N})\}_{k}$ contains at-least $p^{Nc_{1}}$ of the points $\left\{\frac{0}{p^{N}},\frac{1}{p^{N}},\ldots,\frac{p^{N}-1}{p^{N}}\right\}$.

Denote by $A_{N}$ the set of atoms of $P_{p^{N}}$ which contain at least one point from $\{a_{k}x_{N} : k \in \mathbb{N}\}$.
Denote by $B_{N}$ the set of atoms of $P_{p^{N}}$ which contain at least one point from $\{a_{k}y_{N} : k \in \mathbb{N}\}$.
Define $M_{N}$ to be $A_{N}$ if $\lvert A_{N} \rvert\geq \lvert B_{N} \rvert$ and $B_{N}$ otherwise.
By the observation, $\lvert M_{N} \lvert \geq \frac{1}{2}\lvert \Delta_{N} \rvert^{1/2} \geq \frac{1}{2}p^{\frac{1}{2}Nc_{2}}$, for some $c_{2}>0$.
Choose distinct points $z_{i}$ on the orbit closure of $\{p^{n}a_{k}x : n,k\in\mathbb{N}\}$ which lies in distinct atoms of $M_{N}$.

Define the following measures -
\begin{equation}
\label{m-def}
m_{N}=\frac{1}{\lvert M_{N} \rvert}\sum_{i=0}^{\lvert M_{N} \rvert-1}\delta_{z_{i}}
\end{equation}
Where $\delta_{z_{i}}$ are the unit-mass measures defined at the point $z_{i}$.
\begin{lem} \label{lem:large-entropy-measure-initial}
$H_{m_{N}}(P_{p^{N}})\geq c_{2}N\log(p)$.
\end{lem}
\begin{proof} The proof is by the following computation -
\begin{equation}
H(m_{N},P_{p^{N}})=-\sum_{\text{s is an atom of }P_{p^{N}}}m_{N}(s)\log(m_{N}(s))=\log(\lvert M_{N} \rvert)
\end{equation}
By the observation from above, we can deduce the following bound -
\begin{equation}\label{eq:entropy-m_N}
\log(\lvert M_{N} \rvert)\geq c_{2}N\log(p)
\end{equation}
\end{proof}
Notice that the measure $m_{N}$ does not need to be $T_{p}$-invariant.

We apply the following averaging procedure over $T_{p}$ in order to achieve $T_{p}$-invariance of a weak-$*$ limit of those measures.
Define the $T_{p}$ k'th-average of $m_{N}$ as
\begin{equation}\label{m_N^k-def}
m_{N}^{k}=\frac{1}{k}\sum_{i=0}^{k-1}(T_{p}^{i}).m_{N}
\end{equation}

\begin{lem}\label{lem:avg-pos-entropy}
For any $0 < k < c_{2}N$, $H_{m_{N}^{k}}(P_{p^{N}})\geq c_{3}N\log(p)$ for some constant $c_{3}>0$.
\end{lem}
\begin{proof}
By convexity properties of the entropy function we have -
\begin{equation}\label{eq:affinity-entropy}
H_{m_{N}^{k}}\left(P_{p^{N}} \right)\geq \frac{1}{k} \sum_{i=0}^{k-1}H_{T_{p}^{i}.m_{N}}\left(P_{p^{N}} \right). 
\end{equation}
By construction, $m_{N}$ is evenly distributed among more than $p^{c_{2}N}$ different atoms of the partition $P_{p^{N}}$.
Hence, in the worst case, $T_{p}.m_{N}$ is concentrated along $p^{c_{2}N-1}$ different atoms of the partition $P_{p^N}$, as $T_p$ is a $p$-to-$1$ map, therefore $H_{T_{p}.m_{N}}\left(P_{p^{N}} \right) \geq (c_{2}N-1)\log(p)$, and similarly $H_{T_{p}^{i}.m_{N}}\left(P_{p^{N}} \right) \geq (c_{2}N-i)\log(p)$ for every $i=1,\ldots,k$.
Substituting in \eqref{eq:affinity-entropy} we get:
\begin{equation*}
H_{m_{N}^{k}}\left(P_{p^{N}} \right)\geq c_{2}N\log(p)-\frac{k+1}{2}\log(p).
\end{equation*}
As $k$ was assumed to be smaller than $c_{2}N$, the claim follows.
\end{proof}

\begin{prop}\label{prop:mu_N-almost-inv}
$m_{N}^{N^{\delta}}$ is \emph{almost $T_{p}$-invariant}, $|T.m_{N}^{N^\delta}\left(f\right)-m_{N}^{N^\delta}\left(f\right)|=O_{f}(N^{-\delta})$ for any continuous function $f$.
\end{prop}
\begin{proof}
Let $f$ be a continuous function on $\mathbb{R}/\mathbb{Z}$.
\begin{align*}
\left|m_{N}^{N^\delta}(f)-T_{p}.m_{N}^{N^\delta}(f) \right| &\leq \frac{1}{N^\delta}\left|\sum_{i=0}^{N^\delta-1}(T_{p}^{i}).m_{N}(f)-(T_{p}^{i+1}).m_{N}(f) \right| \\ &= \frac{1}{N^\delta}\left|m_{N}(f)-T_{p}^{N}.m_{N}(f) \right|
\\ &\leq \frac{2\|f\|_{\infty}}{N^\delta} \xrightarrow{N\to\infty} 0.
\end{align*}
\end{proof}

We would like to establish a semi-continuity principal for the entropies so that a weak-$*$ limit of those invariant measures will have positive entropy.
Upper semi-continuity of the entropy is known for expansive maps (cf. \cite[Theorem 8.2]{walters} and the discussion in \cite[Section 9]{EKL}), but unfortunately, the arguments which are given there are applicable only in the case the limit is taken from a set of \emph{$T$-invariant measures}.
The $T$-invariance is used to establish a sub-additivity argument, and more precisely, the $T$-invariance allows one to change the scales in the refinements of the partition in an effortless manner. We show how one can change scales in the situation we consider here.

Define $\mu$ to be any weak-$*$ limit of the sequence of probability measures $\{m_{N}^{N^\delta}\}$.
Because the measures $\{m_{N}^{N}\}$ were almost $T_{p}$-invariant, $\mu$ is \emph{$T_{p}$-invariant}.
\begin{lem}[Upper semi-continuity of entropy]\label{lem:semi-continuity}
Under the previous assumptions, we have -
\begin{equation}
\label{eq:sub-additive}
h_{\mu}(T_{p})>0.
\end{equation}
\end{lem}
\begin{proof}
As mentioned before, we are adapting the proof given in \cite[Section 9]{EKL} to the current situation. \\
$P_{p}$ is a \emph{generating partition} for the system $\left(\mathbb{R}/\mathbb{Z},T_{p}\right)$, and its $N$'th refinement under $T_{p}$ is exactly $P_{p^{N}}$.
By Sinai's theorem we can compute the entropy in the following way (cf. \cite[Lemma 9.1, Proposition 9.2]{EKL},for related results) -
\begin{equation}\label{eq:entropy-estimate}
h_{\mu}(T_{p})=h_{\mu}\left(T_{p},P_{p}\right)=\lim_{N\to\infty}\frac{1}{N}H_{\mu}\left(P_{p^{N}}\right).
\end{equation}
By the definition of the Kolmogorov-Sinai entropy, fixing $\varepsilon>0$, for $N_{1}\gg 0$ we have that
\begin{equation}
\frac{1}{N_{1}}H_{\mu}\left(P_{p^{N_1}}\right)\leq h_{\mu}\left(T_{p},P_{p}\right)+\varepsilon/2.
\end{equation}

One would want to compare the following entropies - $H_{m_{N}^{N^\delta}}\left(P_{p^{N_1}} \right)$ and $H_{\mu}\left(P_{p^{N_1}}\right)$ by weak-$*$ convergence of the atoms in partition $P_{p^{N_1}}$.
But we cannot assume that the atoms of the partition $P_{p^{N_1}}$ have $\mu$-null boundaries.
Instead we will look in the modified partition $x+P_{p^{N_{1}}}$, where we translate every atom $p$ of $P_{p^{N_1}}$ by a generic number $x$ modulo $1$.
We argue that there exists such $x$ so that the atoms of the partition $x+P_{p^{N_{1}}}$ have $\mu$-null boundaries.
Otherwise, $\mu$ will have uncountably many atoms, but as a probability measure, the set of atoms of $\mu$ is countable.

Now we claim that the entropies of the partitions $P_{p^{N_{1}}},x+P_{p^{N_{1}}}$ are comparable, more precisely -
\begin{equation}\label{eq:translate-entropy}
\left|H_{\nu}\left(P_{p^{N_1}} \right)-H_{\nu}\left(x+P_{p^{N_1}} \right) \right|\leq O(1),
\end{equation}
 where $\nu$ is either $\mu$ or $m_{N}^{N^\delta}$.

Notice that if a measure $\nu$ is essentially spread uniformly over $p^{H_{\nu}(P)}$ cells in a partition $Q$, where $Q$ is a partition of $[0,1]$ to equally sized intervals, then by translating with the number $x$ the spreading might change, but in a controllable way. In the worst case, by translating with $x$, we might have that half of the translated atoms lose their mass, and the other half have twice their original mass.
Hence -
\begin{equation}
H_{\nu}\left(x+Q \right)\geq \log\left(\frac{1}{2}p^{H_{\nu}(Q)}\right)=H_{\nu}(Q)+\log\left(\frac{1}{2} \right).
\end{equation}
By symmetry, we have the following inequality
\begin{equation}
H_{\nu}\left(Q \right)\geq H_{\nu}\left(x+Q\right)+\log\left(\frac{1}{2} \right).
\end{equation}
Therefore, we have proved~\eqref{eq:translate-entropy}.
By inequality~\eqref{eq:translate-entropy} we have that 
\begin{equation}
\frac{1}{N_{1}}\left|H_{\mu}\left(P_{p^{N_1}}\right)-H_{m_{N}^{N^\delta}}\left(P_{p^{N_1}} \right) \right| \leq \frac{1}{N_{1}}\left|H_{\mu}\left(x+P_{p^{N_1}}\right)-H_{m_{N}^{N^\delta}}\left(x+P_{p^{N_1}} \right) \right|+O(1/N_{1}). 
\end{equation}
By weak-$*$ convergence we have that for every atom $p$ of $x+P_{p^{N_1}}$,
\begin{equation*}
m_{N}^{N^\delta}(p)\to \mu(p).
\end{equation*}
Consequently, for large enough $N$ -
\begin{equation}\label{eq:compare-entropy}
\frac{1}{N_1}\left|H_{m_{N}^{N^\delta}}\left(x+P_{p^{N_{1}}}\right)-H_{\mu}\left(x+P_{p^{N_1}}\right) \right|<\varepsilon/2.
\end{equation}
So we have that
\begin{equation*}
h_{\mu}(T_{p},P_{p})\geq \frac{1}{N_{1}}H_{m_{N}^{N^\delta}}\left(P_{p^{N_1}}\right)-\varepsilon-o(1). 
\end{equation*}
By subadditivity of the entropy, we have that $H_{m_{N}^{N^\delta}}\left(P_{p^{N_1}}\right)\geq H_{m_{N}^{N^\delta}}\left(P_{p^{N}}\right) \geq c_{3}N_{1}\log(p)$ for some $c_3>0$, as long as $N>N_1$. Hence
\begin{equation*}
h_{\mu}(T_{p},P_{p})\geq c_{3}-\varepsilon-o(1),
\end{equation*}
which we can assume is greater than zero for large enough $N$, and $\varepsilon$ small enough.
\end{proof}

Notice that this limit measure $\mu$ is not necessarily $T_{p}$-ergodic, but by taking the ergodic decomposition of $\mu$
\begin{equation*}
\mu=\int_{T_{p}\text{ ergodic measures}}\mu^{\mathcal{E}}d\mathcal{E},
\end{equation*}
we can find an ergodic component with positive entropy.
This completes the proof of Theorem~\ref{thm:inv-measure-positive-entropy} in the prime case.

\subsection{Proof of Theorem~\ref{thm:inv-measure-positive-entropy} - composite case.}\label{sub-composite-theorem}
The proof of the theorem is very similar to the proof of the prime case.
Using proposition~\ref{prop:difference-of-rational-points-in-closure} we construct two sequences of points $\{x_{n} : n\in \mathbb{N}\},\{y_{n} : n\in \mathbb{N}\} \subset \mathbb{R}/\mathbb{Z}$ for which the elements of each sequence belong to the orbit closure of $x$ under the $T_{q}$ map, and satisfy $x_{i}-y_{i}=\ell_{i}/q^{i}$ for each $i\in\mathbb{N}$ where $\ell_{i}\in\{1,\ldots,q^{i}-1\}$ and $\ell_{i+1} = \ell_{i} \mod {q^{i}}$.
By induction we deduce $\gcd(\ell_{i},q)=\gcd (\ell_{1},q)$, we denote $\gcd(\ell_{1},q)$ as $d$.
By the congruence properties of the sequence $\{\ell_{n} : n\in\mathbb{N}\}$, we know that $\gcd(\ell_{n}/d,p)=1$ for every prime number $p$ which divides $q$, hence for every given power $p^{e}$ of $p$ we have $\gcd(\ell_{n}/d,p^{e})=1$ also. By the Chinese reminder theorem we learn that for each integer $n$ $\ell_{n}/d$ is an invertible element modulo $q^{n}$.
Therefore, the linear map induced by multiplication by $\ell_{n}$ on $\mathbb{Z}/q^{n}\mathbb{Z}$ is at-most $d$-to-$1$ map.

Let $\{a_{k} : k \in \mathbb{N} \}$ be a sequence having $h_{p\text{-adic}}\left(\{a_{k}\}\right)=c_{1}>0$ for some fixed prime number $p$ which divides $q$.
Hence for large $N$, using the Chinese reminder theorem, the reduced sequence $\left\{a_{k}^{\left(q^{N}\right)} : k\in \mathbb{N}\right\} \subset\mathbb{Z}/q^{N}\mathbb{Z}$ contains at-least $p^{Nc_{2}}$ elements, for some $c_{2}>0$ constant.
Therefore, $\ell_{N}\cdot\left\{a_{k}^{(q^{N})} : k\in \mathbb{N}\right\}$ contains at-least $p^{Nc_{2}}/d$ elements.

Examining the Minkowski difference of the sets $\overline{\{a_{k}x_{N} : k \in \mathbb{N}\}},\overline{\{a_{k}y_{N} : k \in \mathbb{N}\}}$, we 
deduce that at-least one of those sets, say $\overline{\{a_{k}x_{N} : k \in \mathbb{N}\}}$, intersects at-least $\frac{p^{Nc_{1}}}{2d}$ atoms of the partition $P_{q^N}$. Define $\{z_{n} : n \in\mathbb{N}\}$ to be a set of points sampled from $\overline{\{a_{k}x_{N} : k \in \mathbb{N}\}}$ such that each point lies in a distinct $P_{q^N}$ atom.
Defining $m_{N}$ to be the uniform measure over the set of points $\{z_{n} : n\in\mathbb{N}\}$, we conclude the following analogue of lemma~\ref{lem:large-entropy-measure-initial} -
\begin{lem}
$H_{m_{N}}(P_{q^{N}})\geq c_{2}N\log(p)-c_{2}\log(d) \geq c_{3}N\log(p)$.
\end{lem}

The rest of the proof is verbatim the same as the proof of Theorem~\ref{thm:inv-measure-positive-entropy}.

\section{Proof of Theorem \ref{thm:p-adic}}\label{proof-p-adic}\hypertarget{proof-p-adic}
In this section we will prove Theorem~\ref{thm:p-adic} by using Meiri's charicterization of $p$-Host sequences.

\begin{thm}\label{pre-meiri}
Let $q$ be a fixed integer, $\{a_{m} : m\in \mathbb{N}\}$ a sequence of integers having \emph{positive upper $q$-adic entropy}, and let $\{b_{k} : k\in\mathbb{N}\}$ be a $q$-Host sequence, then for every $x\in\mathbb{R}\backslash\mathbb{Q}$ we have that the set $\{q^{n}a_{m}b_{k}x : n,m,k\in\mathbb{N}\}$ is dense in $\mathbb{R}/\mathbb{Z}$.
\end{thm}

\begin{proof}
By Theorem~\ref{thm:inv-measure-positive-entropy}, there exists a Borel probability measure which is $T_{q}$-invariant and having positive entropy supported inside $\overline{\{q^{n}a_{m}x : n,m\in\mathbb{N}\}}$.
By the definition of a $q$-Host sequence, for $\mu$ almost-every $y$, the sequence $\{b_{k}y : k\in\mathbb{N}\}$ is equidistributed in $\mathbb{R}/\mathbb{Z}$ with respect to the Haar measure, and in-particular, dense.
Approximating such a point with points from $\{q^{n}a_{m}x : n,m\in\mathbb{N}\}$, we deduce the theorem.
\end{proof}

In \cite{meiri}, D. Meiri proved the following theorem - 
\begin{thm}[\cite{meiri},Theorem~$3.2$] Assume that $\{a_{n} : n\in \mathbb{N}\}$ has a smooth $p$-adic interpolation by a function $f$ which has finitely many critical points inside the unit disc, then $\{a_{n} : n\in\mathbb{N}\}$ is a $p$-Host sequence.
Moreover, assume that $q$ is a composite integer, and for each prime number $p$ dividing $q$ the sequence $\{a_{n} : n\in \mathbb{N}\}$ has a smooth $p$-adic interpolation by a function $f$ which has finitely many critical points inside the unit disc, then $\{a_{n} : n\in\mathbb{N}\}$ is a $q$-Host sequence.
\end{thm}
Lindenstrauss' subsequent work (\cite{lindenstrauss2001p}, Theorem~$3.2$, Examples~$3.2,3.3$) generalized the above mentioned theorem, and in-particular relaxed its conditions, requiring only one prime number $p$ which divides $q$ for which $\{a_{n}\}$ has a smooth $p$-adic interpolation in order to conclude that $\{a_{n}\}$ is a $q$-Host sequence, as long as for any other prime $p'$ which divides $q$ either $\{a_{n}\}$ admits a smooth $p'$-adic interpolation or $\|a_{n}\|_{p'} \to 0$. A prototypical example to the above mentioned situation is the sequence $a_{n}=2^{2^{n}}$ and $q=6$, where the sequence is $3$-Host but $\|a_{n}\|_{2} \to 0$. 

Moreover, during the proof of his theorem, Meiri established the following theorem - 
\begin{thm}[\cite{meiri}, Theorem~$3.1$] Assume that $\{a_{n} : n\in \mathbb{N}\}$ has a smooth $p$-adic interpolation by a function $f$ which has finitely many critical points inside the unit disc, then $\overline{h_{p-\text{adic}}}(\{a_{n}\})>0$.
\end{thm}

Combining Meiri's and Lindenstrauss' theorems with Theorem~\ref{pre-meiri} we deduce the following corollary - 
\begin{cor}\label{coro-main-theorem} Let $q$ be an integer and let $\{a_{m} : m \in \mathbb{N}\}, \{b_{k} : k\in\mathbb{N}\}$ be two sequences of integers.
If for some prime number $p$ dividing $q$ the sequence $\{a_{m}\}$ admits a smooth $p$-adic interpolation with only finitely many critical points inside the unit disc, and for some prime number $p'$ which divides $q$ the sequence $\{b_{k} : k\in\mathbb{N}\}$ admits a smooth $p'$-adic interpolation with only finitely many critical points inside the unit disc and for every other prime $p'$ which divides $q$ either $\{b_{k}\}$ admits a smooth $p'$-adic interpolation with only finitely many critical points inside the unit disc or $\|b_{k}\|_{p'} \to 0$, then for any $x\in\mathbb{R}\backslash\mathbb{Q}$ the set $\{q^{n}a_{m}b_{k}x : n,m,k\in\mathbb{N}\}$ is dense in $\mathbb{R}/\mathbb{Z}$.
\end{cor}
This concludes the proof of Theorem~\ref{thm:p-adic}.

\begin{defn}
A \emph{locally $p$-adic analytic} function $f$ is a function defined by power series expansion in some open disc around the origin in a finite extension of $\mathbb{Q}_p$.
\end{defn}
\begin{obs}
By the Formal Substitution Lemma for $p$-adic analytic functions (c.f. \cite{robert-p-adic-book}, Section $6.1.5$), one can compose two locally $p$-adic analytic functions $f(x), g(x)$ where $g(x)$ satisfies $g(0)=0$ and some moderate condition over the growth modulus, and get a locally $p$-adic analytic function.
\end{obs}
\begin{ex}
The sequence $\{ 3^{n} : n\in \mathbb{N} \}$ does not admit a $2$-adic interpolation, as the function $f\left(x\right)=3^{x}=\exp \left(x\log\left(3\right)\right)$ 
where $\exp$ is the $2$-adic exponential is not $2$-adic analytic function defined in the whole unit disc, as the exponential has radius of convergence equals to $2^{-1/2}$ and the function $\log \left(1+x\right)$ has radius of convergence equals to $1$.
By examining a modified version of the function - $f_{2}\left(x\right)=3^{2\cdot x}$, we see that the function $f_{2}$ is indeed a $2$-adic analytic function defined in the whole unit disc, as $val_{2}(2)=1$, therefore see that $f$ is a \emph{locally $2$-adic analytic function}.
In general, the $p$-adic exponential function $\exp_{p}$ only converges in the disc $\lvert z \rvert_{p} < p^{-1/(p-1)}$, by picking a suitable integer $k$ for which $val_{p}(k)-1/(p-1)>0$, we can find sub-sequences of ''exponentially defined'' sequences which admits smooth $p$-adic interpolation.
\end{ex}
In a similar fashion to the example, in view of Meiri's and Lindenstrauss' results, we have the following - 
\begin{prop}
Fix an integer $q$. Assume that a sequence of integers $\{a_{n} : n\in \mathbb{N}\}$ is given by the following formula - $a_{n}=f\left(n\right)$ where $f$ is some \emph{locally $p$-adic analytic} function with finitely many critical values inside some disc in $\mathbb{Q}_{p}$, for some prime number $p$ dividing $q$. Then for a suitable sub-sequence $\{a_{n_k} : k\in\mathbb{N}\} \subset \{a_{n} : n\in\mathbb{N}\}$ we have that $\{a_{n_k} : k\in\mathbb{N}\}$ has positive upper local $q$-adic entropy.
Moreover, if for any other prime $p'$ which divides $q$ either $\{a_{n}\}$ admits a smooth interpolation by locally $p'$-adic analytic function with finitely many critical values inside some disc in $\mathbb{Q}_{p'}$ or $\|a_{n}\|_{p'} \to 0$, then there exists a sub-sequence $\{a_{n_k} : k\in\mathbb{N}\} \subset \{a_{n} : n\in\mathbb{N}\}$ which is $q$-Host sequence.
\end{prop}
\begin{proof}
Define $R$ to be the minimal radius of convergence of the function $f\left(x\right)$ from the different radii of convergence of $f\left(x\right)$ for the various prime numbers $p$ which dividing $q$, in case there is more than one prime $p$ which divides $q$ for which $\{a_{n}\}$ admits a locally $p$-adic analytic interpolating function.
Choosing some integer $S \gg_{q} R$ (for example, one can take $S$ to be the radical of $q$ to the power of $-\log_{p}(R)$ where $p$ is the largest prime dividing $q$) and looking at the sub-sequence $\{a_{S\cdot n} : n\in \mathbb{N} \}$ we deduce that $f\left(S\cdot x\right)$ is a smooth $p$-adic interpolation for this sub-sequence, having only finitely many critical points in the unit disc, and by Meiri's theorem, the sub-sequence $\{a_{S\cdot n} : n\in \mathbb{N} \}$ is having positive upper local $q$-adic entropy.
Notice that for any distinct primes $p,p'$ which divide $q$ we have that $val_{p}(p')=1$ therefore moving to such a sub-sequence would indeed result in a smooth $p$-adic analytic function with only finitely many critical points in the unit disc for any prime $p$ for which $\{a_{n}\}$ admits such interpolation.
In case where $\|a_{n}\|_{p} \to 0$ then obviously $\|a_{Sn}\|_{p} \to 0$ as well, and Lindenstrauss' characterization of $q$-Host sequence implies that $\{a_{Sn}\}$ is a $q$-Host sequence.
\end{proof}

Notice that by the proposition, Corollary~\ref{coro-main-theorem} immediately generalizes to the case where the sequences $\{a_{m} : m \in \mathbb{N}\}, \{b_{k} : k\in\mathbb{N}\}$ are given by interpolation by \emph{locally $p$-adic analytic} functions, as for any $x\in\mathbb{R}\backslash\mathbb{Q}$ one can consider the subset $\{q^{n}a_{S\cdot m}b_{S'\cdot k}x : n,m,k\in\mathbb{N}\}\subset \{q^{n}a_{m}b_{k}x : n,m,k\in\mathbb{N}\}$, where $S,S'$ are the integers which have been computed in the previous proposition, which is dense by the corollary.

Combining the proposition with the observation regarding the $p$-adic substitution lemma, one can deduce the following generalization of Furstenberg's density result for a sparse sequence - 
\begin{cor}[Sparse density theorem] For any irrational number $x\in\mathbb{R}\backslash\mathbb{Q}$ and two non-constant polynomials $p_{1}\left(x\right),p_{2}\left(x\right)$ with integer coefficients, the set \\ $\left\{2^{n}3^{3^{\cdot^{\cdot^{\cdot^{3^{p_{1}(m)}}}}}}3^{3^{\cdot^{\cdot^{\cdot^{3^{p_{2}(k)}}}}}}x : n,m,k\in\mathbb{N}\right\}$ is dense in $\mathbb{R}/\mathbb{Z}$.
\end{cor}

Notice that the number of distinct elements of the sequence \\ $\left\{2^{n}3^{3^{\cdot^{\cdot^{\cdot^{3^{p_{1}(m)}}}}}}3^{3^{\cdot^{\cdot^{\cdot^{3^{p_{2}(k)}}}}}} : n,m,k\in\mathbb{N}\right\}$ which are contained in the interval $\lbrack 1,N\rbrack$ is about the size of $\log\left(N\right)$ multiplied by $\log^{\left(i\right)}\left(N\right)$ for some $i>1$, hence we have shown the following result - 

\begin{cor}\label{cor:universally-densifying}
Let $r:\mathbb{N}\to\mathbb{N}$ be an increasing function, which grows at-least like $\log^{\left(i\right)}\left(N\right)$ for some $i>1$, then there exists a sequence $\{a_{n} : n\in\mathbb{N}\}$ of integers which satisfy $\left\vert\{a_{n} \}\cap \lbrack 1,N\rbrack \right\vert\leq r\left(N\right)\log\left(N\right)$ such that for any irrational $x\in\mathbb{R}\backslash\mathbb{Q}$, the set $\{a_{n}x : n\in \mathbb{N}\}$ is dense in $\mathbb{R}/\mathbb{Z}$.
\end{cor}

\bibliographystyle{plain}
\bibliography{bibliography-Furstenberg}

\end{document}